\documentclass[12pt,a4paper]{amsart} 
\usepackage{amsmath,amsfonts,amssymb,amsthm,amscd,slashed}
\usepackage{epsfig}
\pdfoutput=1

\usepackage{tikz}

\usepackage{xypic}
\usepackage{hyperref}
\usepackage{JK}                

\newtheorem{thm}{Theorem}

\newtheorem{lma}[thm]{Lemma} \newtheorem{prop}[thm]{Proposition}
\newtheorem{defn}[thm]{Definition}

\newtheorem{rem}[thm]{Remark}

\def\1{\mathbb{I}}

\def\C{\mathbb{C}}
\DeclareMathOperator{\Cl}{Cl}

\DeclareMathOperator{\dom}{Dom}

\def\E{\mathcal{E}}

\DeclareMathOperator{\End}{End}

\def\half{\tfrac{1}{2}}

\def\S{\mathbb{S}}

\def\T{\mathbb{T}}

\def\tilde{\widetilde}

\title[Riemannian submersions, factorization of Dirac operators]{Riemannian submersions and factorization of Dirac operators}
\author{Jens Kaad and Walter D. van Suijlekom}

\address{Department of Mathematics and Computer Science, Syddansk Universitet, Campusvej 55, 5230, Odense M, Denmark}
\email{jenskaad@hotmail.com}
\address{Institute for Mathematics, Astrophysics and Particle Physics, Radboud University Nijmegen, Heyendaalseweg 135, 6525 AJ Nijmegen, The Netherlands}
\email{waltervs@math.ru.nl}
\date{\today}


\begin{document}
\maketitle

\begin{abstract}
We establish the factorization of Dirac operators on Riemannian submersions of compact spin$^c$ manifolds in unbounded $KK$-theory. More precisely, we show that the Dirac operator on the total space of such a submersion is unitarily equivalent to the tensor sum of a family of Dirac operators with the Dirac operator on the base space, up to an explicit bounded curvature term. Thus, the latter is an obstruction to having a factorization in unbounded $KK$-theory. We show that our tensor sum represents the bounded $KK$-product of the corresponding $KK$-cycles and connect to the early work of Connes and Skandalis. 

\end{abstract}

\section{Introduction}
\label{sect:intro}
Noncommutative geometry \cite{C94} is a vast generalization of differential geometry to the quantum world. However, many of its successful applications are in ordinary, commutative differential geometry. For instance, it turned out that Kasparov's bivariant $K$-theory provided the right context for index theory which naturally extends to foliations \cite{C82}. Here, a central role is played by the shriek or wrong-way map $f!$ associated to any ($K$-oriented) smooth map $f:X \to Y$ between smooth manifolds. It is an element in the bivariant $K$-theory $KK(C(X),C(Y))$ of Kasparov \cite{Kas} and is defined using the principal symbol of a suitably defined pseudodifferential operator of order $0$. We refer to \cite{C82,CS84} (and \cite{HS87} for the general case of maps between foliations) for full details. For the special case that $Y$ is a point the shriek map is the fundamental class $[X]$ of the manifold in the $K$-homology group $KK(C(X),\C)$. The wrong-way functoriality of the shriek map was stated as a problem in \cite{C82} and proven in \cite{CS84}. It says that for two maps $f : X \to Y$ and $g:Y \to Z$ we have
$$
(g \circ f)! = f! \hot_{C(Y)} g!
$$
where $\hot_{C(Y)}$ denotes the internal Kasparov product in $KK$-theory.

Now, for the special case that $Z$ consist of a single point, this implies that the fundamental class of $X$ factorizes in $KK$-theory as follows:
\begin{equation}
\label{eq:fact}
[X] = f! \hot_{C(Y)} [Y].
\end{equation}
Moreover, if $f:X \to Y$ is a submersion, the shriek map $\pi!$ can equivalently be described by a family of pseudodifferential operators of order zero acting on the fibers of $f$ \cite[Proposition 2.9]{CS84}. A natural first question that arises is whether (and under which conditions) one can explicitly find unbounded $KK$-cycles (as introduced in the early 1980s by Connes and Baaj--Julg in \cite{BJ83}) that represent the respective classes $[X], [Y]$ and $f!$. If $X$ and $Y$ are spin$^c$ manifolds, it is clear \cite{C85} that the Dirac operators $D_X$ and $D_Y$ represent the corresponding fundamental classes. Moreover, if $f: X \to Y$ is a Riemannian submersion, one expects that there is a family of Dirac operators $\{ S_y \}$ on the fibers $f^{-1}(\{y\})$ (for all $y \in Y$) that gives an unbounded representative of the class $f!$ in $KK(C(X),C(Y))$. The immediate next question is then whether we can write the factorization formula \eqref{eq:fact} as a tensor sum of these unbounded $KK$-cycles, that is to say, whether in some sense
\begin{equation}
\label{eq:tensor-sum}
D_X = S \otimes 1 + 1 \otimes_\nabla D_Y,
\end{equation}
with $\nabla $ a suitable connection on the bundle of vertical spinors. 

Even though the above are natural questions, and their affirmative answers appear to be folklore to practitioners of (unbounded) $KK$-theory, we have not found a detailed written account on it in the literature. Of course, the bounded $KK$-cycles that enter \cite{CS84} are very much differential in nature and essentially already dictate the form of the corresponding unbounded $KK$-cycles. The work of Bismut \cite{Bis86} (see also \cite[Chapter 10]{BGV92} and \cite[Chapter 4]{GLP99}) comes very close, at least in spirit, but does not connect to $KK$-theory (even though that would provide the right context for the Atiyah--Singer index theorem for families). Also, the mere existence of the unbounded $KK$-cycles and the validity of the tensor sum as an unbounded representative of the internal Kasparov product would follow from \cite{MR15}. There, the authors work with the unbounded form of the internal $KK$-product that was the topic of \cite{Mes09b, KL13} and use a theorem by Kucerovsky \cite{Kuc97} to check if an unbounded $KK$-cycle is a representative of the internal product of (the bounded classes of) two other unbounded $KK$-cycles. Without questioning the power of their approach, the geometric context sketched above makes it most natural to explicitly construct the unbounded $KK$-cycles in terms of (families of) Dirac operators and establish for them an explicit tensor sum decomposition as in Equation \ref{eq:tensor-sum}. Moreover, as we will see below (Theorem \ref{thm:fact}) there is an obstruction to having an exact factorization of unbounded $KK$-cycles given by a (bounded) curvature form. In fact, this is the key point of working with unbounded $KK$-cycles instead of merely their homotopy classes in bounded $KK$-theory: one captures the metric aspect of the geometry which is encoded by an unbounded operator. From this point of view it is thus not surprising that curvature enters in the (unbounded product) formula for submersions.

\medskip

In this paper, we will present in full detail the construction of an unbounded representative for $\pi!$ in the case of a Riemannian submersion $\pi:M \to B$ of spin$^c$ manifolds $M$ and $B$. It is given by a family of Dirac operators on the fibers of $\pi$ and is defined in terms of a vertical Clifford connection acting on a vertical Clifford module.

After we have shown that the tensor product of the vertical Clifford module with the Hilbert space of $L^2$-spinors on $B$ is unitarily equivalent to the Hilbert space of $L^2$-spinors on $M$, we state our main result in Theorem \ref{thm:fact}. We establish that the tensor sum $S \otimes 1 + 1 \otimes_\nabla D_B$ is unitarily equivalent to $D_M$, up to an explicit (bounded) curvature term. As already mentioned, such a curvature term cannot be obtained by analyzing the bounded $KK$-product, but only arises as an obstruction to having an exact factorization of unbounded $KK$-cycles.

We exemplify our results by homogeneous spaces for simply-connected compact semisimple Lie groups. Note that this class and the corresponding factorization in $KK$-theory is also being studied in \cite{CM16}. As a special case in this class we obtain the factorization in $KK$-theory of the Dirac operator on the total space of the Hopf fibration which was previously obtained in \cite{BMS13}. We also explain the term  $-\frac12$ that appeared in the tensor sum in {\it loc.~cit.~}as coming from the curvature of the Hopf fibration. Finally, our construction generalizes the projectability results on circle and torus principal bundles of \cite{DS10,DSZ13,DZ13} and the $KK$-factorization results of torus principal bundles of \cite{FR15}.

\subsection*{Acknowledgements}

We would like to thank Alain Connes and Nigel Higson for useful suggestions. We are grateful to Simon Brain, Bram Mesland and Adam Rennie for fruitful discussions. The first author also wants to thank Matthias Lesch for a sequence of inspiring conversations back in 2012 concerning the possibility of applying unbounded $KK$-theory techniques in the context of fiber bundles.

\section{Riemannian submersions}
We start by giving a brief overview of Riemannian submersions, referring to \cite[Chapter III.D]{GHL87} and \cite[Chapter 10.1]{BGV92} for more details. 

Let $M$ and $B$ be compact Riemannian manifolds without boundaries and let $\pi:M \to B$ be a smooth and surjective map. We let $\sX(M)$ and $\sX(B)$ denote the vector fields on $M$ and $B$ and the hermitian forms on the vector fields coming from the Riemannian metrics are denoted by $\inn{\cd,\cd}_M$ and $\inn{\cd,\cd}_B$. The forms on $M$ and $B$ are denoted by $\Om(M) = \op_{j = 1}^{\dim(M)} \Om^j(M)$ and $\Om(B) = \op_{j = 1}^{\dim(B)} \Om^j(B)$, respectively.

We obtain a $C^\infty(M)$-module homomorphism:
\[
\begin{split}
& d \pi : \sX(M) \to \sX(B) \ot_{C^\infty(B)} C^\infty(M) \\
& d\pi(X)(f) := X(f \ci \pi) \, , \, \, f \in C^\infty(B)
\end{split}
\]
and we equip the $C^\infty(M)$-module $\sX(B) \ot_{C^\infty(B)} C^\infty(M)$ with the hermitian form
\[
\begin{split}
& \inn{\cd , \cd} : \sX(B) \ot_{C^\infty(B)} C^\infty(M) \ti \sX(B) \ot_{C^\infty(B)} C^\infty(M)
\to C^\infty(M) \\
& \inn{ Y \ot f, Z \ot g} := \ov f \cd ( \inn{Y,Z}_B \ci \pi ) \cd g
\end{split}
\]

We say that $\pi : M \to B$ is a {\em Riemannian submersion} when $d \pi$ is surjective and
\[
d \pi : (\ker d \pi)^\perp \to \sX(B) \ot_{C^\infty(B)} C^\infty(M)
\]
is an isometric isomorphism, where 
\[
(\ker d\pi)^\perp := \big\{ X \in \sX(M) \mid \inn{X,Y}_M = 0 \, , \, \, \forall Y \in \ker(d\pi) \big\}
\]
is the orthogonal complement of $\ker d\pi$.

A vector field $X$ on $M$ is called {\em vertical} if $d \pi(X) = 0$, and {\em horizontal} if $X \in (\ker d \pi)^\perp$. We also write $\sX_V(M):=\ker d \pi$ and $\sX_H(M):= (\ker d \pi)^\perp$ so that we have a direct sum decomposition
\[
\sX(M) \cong \sX_V(M) \oplus \sX_H(M).
\]
Moreover, $\sX_H(M) \cong \sX(B) \ot_{C^\infty(B)} C^\infty(M)$, isometrically. We denote by $P = 1 - (d \pi)^* (d \pi) : \sX(M) \to \sX(M)$ the orthogonal projection onto $\sX_V(M)$, where $(d \pi)^*$ is the adjoint of $d \pi$ with respect to the hermitian forms on $\sX(M)$ and $\sX(B) \ot_{C^\infty(B)} C^\infty(M)$. 

If $Y$ is a vector field on $B$, the unique vector field $Y_H$ on $M$ such that $d \pi (Y_H) = Y \ot 1$ is called the {\em horizontal lift} of $Y$, thus $Y_H := (d\pi)^*(Y \ot 1)$.

In what follows, we will assume that $\pi:M \to B$ is a Riemannian submersion.

We remark that the Lie-bracket $[\cd,\cd]$ of vector fields on $M$ restricts to a Lie-bracket on the vertical vector fields $\sX_V(M)$ but that the same need not be true for the horizontal vector fields. We record the following:
%

\begin{lma}
\label{lma:hor-vert}
If $X$ is a vertical vector field on $M$ and $Y_H$ a horizontal lift of a vector field $Y$ on $B$, then $[X,Y_H]$ is vertical.
\end{lma}
\proof
This follows since
\[
\begin{split}
d \pi( [X,Y_H])(f) & = [X,Y_H](f \ci \pi) = X\big( Y_H(f \ci \pi) \big) \\
& = X\big( (Y \ot 1)(f) \big) = X\big( Y(f) \ci \pi ) = 0
\end{split}
\]
for all $f \in C^\infty(B)$.
\endproof

Recall that on $M$ and $B$ there are Levi-Civita connections $\nabla^M$ and $\nabla^B$, respectively. They are metric and torsionfree, and as a result they satisfy Koszul's formula:
\begin{align}
\label{eq:koszul}
2 \inn{\nabla^M_X(Y),Z}_M & = \inn{[X,Y],Z}_M - \inn{[Y,Z],X}_M + \inn{[Z,X],Y}_M  \\
& \qquad + X(\inn{Y,Z}_M) + Y (\inn{Z,X}_M) - Z( \inn{X,Y}_M),\nonumber
\end{align}
for all \emph{real} vector fields $X,Y,Z \in \sX(M)$ and a similar formula holds for $\nabla^B$. 

\begin{lma}
\label{lma:levi-civitas}
Let $\nabla^M$ and $\nabla^B$ be the Levi-Civita connections on $M$ and $B$, respectively. If $X_H$ and $Y_H$ are horizontal lifts of $X$ and $Y$ in $\sX(B)$, respectively, then
\[
\nabla_{X_H}^M (Y_H) = \left( \nabla_X^B(Y) \right)_H + \half P [X_H,Y_H].
\]
\end{lma}
\proof
Without loss of generality we may assume that $X$ and $Y$ are real vector fields on $B$.

Using Koszul's formula for the Levi-Civita connections, we determine both the vertical and horizontal part of the left-hand-side. First, for a real vertical vector field $Z$ we have
\begin{align*}
& 2 \inn{ \nabla_{X_H}^M(Y_H),Z}_M \\ 
&\q =  \inn{ [X_H,Y_H],Z }_M - \inn{[Y_H,Z],X_H}_M + \inn{[Z,X_H],Y_H}_M \\
& \qq + X_H ( \inn{Y_H,Z}_M) + Y_H( \inn{Z,X_H}_M) - Z( \inn{X_H,Y_H}_M)\\
& \q = \inn{[X_H,Y_H],Z}_M.
\end{align*}
We have used Lemma \ref{lma:hor-vert}, the fact that horizontal and vertical parts are orthogonal, and that $\inn{X_H,Y_H}_M = \inn{X,Y}_B \ci \pi$ so that a vertical $Z$ vanishes on it.

Next, suppose $Z_H$ is a horizontal lift of a real vector field $Z$ on $B$. Then we may use that for any horizontal lifts $X_H,Y_H$,
\[
Z_H( \inn{X_H,Y_H}_M ) = Z ( \inn{X,Y}_B ) \ci \pi.
\]
Together with the fact that $[X_H,Y_H] - [X,Y]_H$ is vertical, this yields
\[
\inn{\nabla_{X_H}^M(Y_H),Z_H}_M = \inn{ \nabla_X^B(Y),Z } \ci \pi = \inn{(\nabla_X^B(Y))_H,Z_H}_M .\qedhere
\]
\endproof
%

We will now define metric connections on the vertical and horizontal vector fields. Using the Levi-Civita connection on $M$ we can define a metric connection $\nabla^{V}$ on $\sX_V(M)$ by
\[
\nabla^{V}_X := P \circ \nabla^M _X \circ i : \sX_V(M) \to  \sX_V(M) \q X \in \sX(M)
\] 
in terms of the orthogonal projection $P : \sX(M) \to \sX(M)$ and the inclusion $i : \sX_V(M) \to \sX(M)$. 

On the horizontal vector fields we define a metric connection using the pullback connection on $\sX(B) \ot_{C^\infty(B)} C^\infty(M) \cong \sX_H(M)$. The pullback connection $\pi^* \nabla^B$ is defined by
\begin{equation}\label{eq:pullconn}
(\pi^* \nabla^B)_X (Y \otimes f) = Y \ot X(f) + \wit{\nabla^B}_{d\pi(X)} (Y) \cd f
\end{equation}
for all $X\in \sX(M)$, $Y \in \sX(B)$ and $f \in C^\infty(M)$, where
\[
\wit{\nabla^B}_{(Z \ot g)} (Y) := \Na^B_Z(Y) \ot g \in \sX(B) \otimes_{C^\infty(B)} C^\infty(M)
\]
for all $Z,Y \in \sX(B)$, $g \in C^\infty(M)$. The pullback connection $\pi^* \Na^B$ is a metric connection on $\sX(B) \ot_{C^\infty(B)} C^\infty(M)$. We then define the metric connection $\Na^H$ on $\sX_H(M)$ by
\[
\Na^H_X := (d \pi)^* (\pi^* \Na^B)_X (d \pi) : \sX_H(M) \to \sX_H(M) \q X \in \sX(M).
\]
%

Note that this implies that the following holds
\begin{equation}\label{eq:hor}
\Na^H_X (Y_H) = (d \pi)^* \big( \wit{\nabla^B}_{d\pi(X)} (Y) \big)
\end{equation}
for all $Y \in \sX(B)$ and all $X \in \sX(M)$.


On the direct sum $\sX_V(M) \oplus \sX_H(M)$ we may combine the above two connections $\nabla^{V}$ and $\pi^* \nabla^B$ to define the {\em direct sum connection} on $\sX(M)$ as
\[
\nabla^\oplus = \nabla^{V} \oplus \nabla^H.
\]
This connection is metric, but might have torsion in general. In \cite{Bis86} Bismut ({\em cf.} \cite[Section 10.1]{BGV92} compared the direct sum connection to the Levi-Civita connection on $M$, finding that the difference can be expressed in terms of a three-tensor $\omega \in \Omega^1(M) \ot_{C^\infty(M)} \Om^2(M)$ defined as follows:
\begin{defn}\label{defn:tensors}
We introduce:
\begin{enumerate}
\item The {\em second fundamental form} 
\[
S \in \Om^1(M) \otimes_{C^\infty(M)} \Om^1(M) \otimes_{C^\infty(M)} \Om^1(M)
\]
defined by
\[
S(X,Y,Z) := \binn{\nabla_{(1-P)Z}^V (PX) - [ (1-P)Z, PX],PY}_M
\]
for real vector fields $X,Y,Z \in \sX(M)$.
\item The {\em curvature of the fiber bundle $\pi : M \to B$} is the element $\Om \in \Om^2(M) \ot_{C^\infty(M)} \Om^1(M)$ given by 
\[
\Omega(X,Y,Z) := -\binn{[ (1-P)X,(1-P)Y], PZ}_M
\]
for real vector fields $X,Y,Z \in \sX(M)$.
\item The tensor $\omega \in \Omega^1(M) \ot_{C^\infty(M)} \Om^2(M)$ defined by
\begin{multline*}
\omega(X)(Y,Z) = S( X,Z,Y) -S(X, Y, Z) \\
+ \half \Omega(X,Z, Y) - \half \Omega ( X,Y,Z) 
+ \half \Omega(Y,Z,X)
\end{multline*}
for $X,Y,Z \in \sX(M)$. 
\end{enumerate}
\end{defn}

We provide a more tangible formula for the second fundamental form:


\begin{prop}
\label{prop:expression-tensor}
We have that
\[
2 S(X,Y,Z) = Z (\inn{X,Y}_M) - \inn{[Z,X],Y}_M - \inn{[Z,Y],X}_M
\]
for all real vector fields $X,Y \in \sX_V(M)$ and $Z \in \sX_H(M)$. In particular, we have the symmetry $S(X,Y,Z) = S(Y,X,Z)$ for all $X,Y,Z \in \sX(M)$.
\end{prop}
\proof
This follows from Koszul's formula for the Levi-Civita connection $\nabla^M$, after restriction to the vertical vector fields using the orthogonal projection $P$. Indeed, for real vertical vector fields $X,Y$ and a real horizontal vector field $Z$ we have that
\begin{align*}
2 \inn{\nabla_Z^{V} X ,Y}_M = \inn{[Z,X],Y}_M + \inn{[Y,Z],X}_M + Z( \inn{X,Y}_M)
\end{align*}
from which the result follows at once. 
\endproof

\begin{prop}[Bismut \cite{Bis86}]\label{p:omedir}
The Levi-Civita connection $\nabla^M$ is related to the direct sum connection $\nabla^\oplus$ by the following formula 
\[
\inn{\nabla_X^M Y,Z}_M = \inn{\nabla_X^\oplus Y,Z}_M + \omega(X)(Y,Z).
\]
\end{prop}
\proof
We follow the proof in \cite[Proposition 10.6]{BGV92}. Without loss of generality we may assume that $X,Y,Z \in \sX(M)$ are all real. We repeatedly use Lemma \ref{lma:hor-vert} and the fact that $\inn{\nabla_X^M Y,Z}_M - \inn{\nabla_X^\oplus Y,Z}_M$ is a tensor in $X,Y,Z$ which is anti-symmetric in $Y$ and $Z$ because both $\nabla^M$ and $\nabla^\oplus$ are metric. We proceed by a case-by-case study of whether $X,Y$ and $Z$ are horizontal or vertical. Moreover, tensoriality makes it sufficient to consider in the horizontal case only horizontal lifts of vector fields on $B$.

\begin{enumerate}
\item If $X,Y,Z$ are all horizontal lifts then $\nabla^\oplus$ reduces to $\nabla^H$ and the result readily follows from Lemma \ref{lma:levi-civitas} and \eqref{eq:hor} since then also $\omega(X)(Y,Z) =0$.
\item If both $Y$ and $Z$ are vertical then $\inn{\nabla_X^M Y,Z}_M$  and $\inn{\nabla_X^\oplus Y,Z}_M$ both coincide with $\inn{\nabla^{V}_X Y,Z}_M$ and $\omega(X)(Y,Z) = 0$.
\item If $X$ and $Y$ are vertical and $Z$ is horizontal, then $\inn{\nabla_X^\oplus Y,Z}_M$ vanishes whereas
\begin{align*}
\inn{\nabla_X^M Y,Z}_M &= -\half \inn{[Y,Z],X}_M + \half \inn{[Z,X],Y}_M - \half Z( \inn{X,Y}_M ) \\
&= -S(X,Y,Z) = \omega(X)(Y,Z).
\end{align*}
\item If $X$ and $Y$ are horizontal lifts and $Z$ is vertical, then $ \inn{\nabla^\oplus_X Y ,Z}_M = 0$ and Lemma \ref{lma:levi-civitas} implies that
\begin{align*}
\inn{\nabla^M_X Y,Z}_M &= \half \inn{[X,Y],Z}_M\\
& = -\half \Omega(X,Y,Z) = \omega(X)(Y,Z).
\end{align*}
\item Finally, if $X$ is vertical and if $Y$ and $Z$ are horizontal lifts then $\nabla^\oplus_X Y = \Na^H_X Y =0$ by \eqref{eq:hor} since $d\pi(X) =0$, and Lemma \ref{lma:hor-vert} gives that
\begin{align*}
\inn{\nabla_X^M Y,Z}_M &= \half \inn{[X,Y],Z}_M - \half \inn{[Y,Z],X}_M + \half \inn{[Z,X],Y}_M  \\
&= \half \Omega(Y,Z,X) = \omega(X)(Y,Z) .\qedhere
\end{align*}
\end{enumerate}
\endproof

\section{Spin geometry and Clifford modules}
\label{sect:clifford-mod}
We now assume that the Riemannian manifolds $M$ and $B$ that enter the Riemannian submersion $\pi:M \to B$ are even-dimensional $\Tex{spin}^c$ manifolds. Hence letting $\Cl(M)$ denote the complex Clifford algebra generated by the vector fields on $M$, we have a $\zz/2\zz$-graded finitely generated projective module $\sE_M$ over $C^\infty(M)$ together with a hermitian form
\[
\inn{\cd,\cd}_{\sE_M} : \sE_M \ti \sE_M \to C^\infty(M)
\]
and an isomorphism of $\zz/2\zz$-graded $*$-algebras
\[
c : \Cl(M) \to \Tex{End}_{C^\infty(M)}(\sE_M).
\]
The grading operator on $\sE_M$ is denoted by $\ga_M : \sE_M \to \sE_M$. We fix similar data and notation for the $\Tex{spin}^c$-manifold $B$.
%
%
%
%

We choose even metric connections $\nabla^{\sE_M}$ and $\nabla^{\sE_B}$ on $\sE_M$ and $\sE_B$ which are Clifford connections for the Levi-Civita connections on $M$ and $B$, respectively. Thus, they satisfy the relation
\begin{align}
\label{eq:clicon}
\nabla^{\sE_M}_X (c(Y)\cdot s) &= c(Y) \cdot \nabla^{\sE_M}_X(s) + c(\nabla^M_X(Y))\cdot s,
\intertext{for $X,Y \in \sX(M)\,  ,  \, \, s \in \sE_M$, and}
\nabla^{\sE_B}_Z (c(W)\cdot t) &= c(W) \cdot \nabla^{\sE_B}_Z(t) + c(\nabla^B_Z(W))\cdot t,\nonumber
\end{align}
for $Z, W \in \sX(B) \, , \, \, t \in \sE_B$. These even metric Clifford connections are unique up to a purely imaginary one-form.
\medskip

We will now pull back the Clifford action and the Clifford connection on $B$ to a horizontal Clifford action and Clifford connection on $M$. The \emph{horizontal spinor module} is defined as the pullback of the spinor module on $B$:
\[
\sE_H := \sE_B \ot_{C^\infty(B)} C^\infty(M).
\]
The horizontal spinor module is $\zz/2\zz$-graded with grading operator $\ga_H := \ga_B \ot 1$ and it has a hermitian form defined by
\[
\inn{\cd,\cd}_{\sE_H} : \sE_H \ti \sE_H \to C^\infty(M) 
\q \inn{s \ot f, t \ot g}_{\sE_H} := \ov f \cd (\inn{s,t}_{\sE_B} \ci \pi) \cd g.
\]
We let $\Cl_H(M) \su \Cl(M)$ denote the $\zz/2\zz$-graded $*$-subalgebra generated by the horizontal vector fields and refer to it as the \emph{horizontal Clifford algebra}. The \emph{horizontal Clifford action} is then defined by
\[
c_H : \Cl_H(M) \to \Tex{End}_{C^\infty(M)}(\sE_H) \q c_H(X) := \wit c\big(d \pi (X) \big)
\]
for all $X \in \sX_H(M)$, where 
\[
\wit c( Y \ot f)(t \ot g) := c(Y)(t) \ot f \cd g \q Y \in \sX(B) \, , \, \, f \in C^\infty(M)
\]
for all $t \in \sE_B$, $g \in C^\infty(M)$. We equip $\sE_H$ with the pullback connection 
\[
\Na^{\sE_H} := \pi^* \Na^{\sE_B}
\]
defined as in \eqref{eq:pullconn}. The connection $\Na^{\sE_H}$ is then even and metric and it is compatible with the horizontal Clifford action in the sense that
\begin{equation}\label{eq:horclicon}
[ \Na^{\sE_H}_X, c_H(Y) ] = c_H( \Na^H_X(Y) ) \q X \in \sX(M) \, , \, \, Y \in \sX_H(M).
\end{equation}

\begin{lma}\label{l:horiclif}
The horizontal Clifford action 
\[
c_H : \Cl_H(M) \to \Tex{End}_{C^\infty(M)}(\sE_H)
\]
is an isomorphism of $\zz/2\zz$-graded $*$-algebras.
\end{lma}
\proof
Since the $C^\infty(B)$-module $\sE_B$ is finitely generated projective, it follows that we have an isomorphism
\[
\Tex{End}_{C^\infty(B)}(\sE_B) \ot_{C^\infty(B)} C^\infty(M) \cong \Tex{End}_{C^\infty(M)}(\sE_H)
\]
of $\zz/2\zz$-graded $*$-algebras. Indeed, any simple tensor $T \ot f$ with $T \in \Tex{End}_{C^\infty(B)}(\sE_B)$ and $f \in  C^\infty(M)$ acts on $\sE_H = \sE_B \ot_{C^\infty(B)} C^\infty(M)$ via the assignment $(T \ot f)(s \ot g) := T(s) \ot f \cd g$. Furthermore, using that $\sE_B$ is a spinor module on $B$ we obtain the isomorphism
\[
\Cl(B) \ot_{C^\infty(B)} C^\infty(M) \cong \Tex{End}_{C^\infty(B)}(\sE_B) \ot_{C^\infty(B)} C^\infty(M)
\]
of $\zz/2\zz$-graded $*$-algebras. Finally, since $\sX_H(M) \cong \sX(B) \otimes_{C^\infty(B)} C^\infty(M)$, isometrically, it follows that
\[
\Cl_H(M)  \cong \Cl(B) \otimes_{C^\infty(B)} C^\infty(M)
\]
as $\zz/2\zz$-graded $*$-algebras.
\endproof

We will now focus on constructing a vertical Clifford action and Clifford connection following \cite[Section 10.2]{BGV92} (to some extent). We let $\Cl_V(M) \su \Cl(M)$ denote the $\zz/2\zz$-graded $*$-subalgebra generated by the vertical vector fields and we refer to it as the \emph{vertical Clifford algebra}.

\begin{defn}
The {\em vertical spinor module} is defined to be the $\zz/2\zz$-graded $C^\infty(M)$-module
\[
\sE_V := \sE_H^* \ot_{\Cl_H(M)} \sE_M \, , 
\]
with $\zz/2\zz$-grading operator $\ga_V := \ga_H^* \ot \ga_M$, where $\ga_H^*$ is the $\zz/2\zz$-grading operator on the dual module $\sE_H^*$ induced by the $\zz/2\zz$-grading on $\sE_H$. The {\em vertical Clifford action} is defined by
\[
c_V : \Cl_V(M) \to \Tex{End}_{C^\infty(M)}(\sE_V) \q c_V(X) := \ga_H^* \ot c(X)
\]
for all $X \in \sX_V(M)$. Furthermore, we have the hermitian form on $\sE_V$ defined by
\[
\binn{ \bra{\xi_1} \otimes s_1,  \bra{\xi_2} \otimes s_2}_{\sE_V}
= \binn{ s_1, \big( \ket{\xi_1}\bra{\xi_2 } \big) (s_2)}_{\sE_M}
\]
for $\bra{\xi_i} \in \sE_H^*$ and $s_i \in \sE_M$, $i=1,2$. The element $\ket{\xi_1}\bra{\xi_2} \in \End_{C^\infty(M)}(\sE_H)$ acts on $\sE_M$ according to the identifications
\[
\End_{C^\infty(M)}(\sE_H) \cong \Cl_H(M) \su \End_{C^\infty(M)}(\sE_M).
\]
\end{defn}

The reason for calling this module on $M$ the ``vertical spinor module'' lies in the following two results:

\begin{prop}\label{l:vertclif}
The vertical Clifford action 
\[
c_V : \Cl_V(M) \to \End_{C^\infty(M)}(\sE_V)
\]
is an isomorphism of $\zz/2\zz$-graded $*$-algebras. In particular, we obtain a $K$-orientation of the Riemannian submersion $\pi : M \to B$, see \cite[Section $10$, $(3)$]{C82} and \cite[Appendix B]{CS84}.
\end{prop}
\begin{proof}
Since $\sE_H^*$ and $\sE_M$ are finitely generated projective modules over $\Cl_H(M)$ and $C^\infty(M)$, respectively, we have the isomorphisms:
\[
\begin{split}
\End_{C^\infty(M)}(\sE_V) 
& \cong \sE_V \ot_{C^\infty(M)} \sE_V^* \\
& \cong \sE_H^* \ot_{\Cl_H(M)} \sE_M \ot_{C^\infty(M)} \sE_M^* \ot_{\Cl_H(M)} \sE_H \\
& \cong \sE_H^* \ot_{\Cl_H(M)} \End_{C^\infty(M)}(\sE_M) \ot_{\Cl_H(M)} \sE_H
\end{split}
\]
of $\zz/2\zz$-graded $*$-algebras. Moreover, since $\sE_M$ defines a $\Tex{spin}^c$-structure on $M$ we have that 
\[
\End_{C^\infty(M)}(\sE_M) \cong \Cl(M) \cong \Cl_H(M) \hot_{C^\infty(M)} \Cl_V(M)
\]
where the tensor product on the right-hand side is graded. Now, from Lemma \ref{l:horiclif} we know that $\Cl_H(M) \cong \sE_H \ot_{C^\infty(M)} \sE_H^*$ and hence that
\[
\begin{split}
& \sE_H^* \ot_{\Cl_H(M)} \End_{C^\infty(M)}(\sE_M) \ot_{\Cl_H(M)} \sE_H \\
& \q \cong \sE_H^* \ot_{\Cl_H(M)} \big( \Cl_H(M) \hot_{C^\infty(M)} \Cl_V(M) \big) \ot_{\Cl_H(M)} \sE_H \\
& \q \cong \Cl_V(M)
\end{split}
\]
where the last isomorphism of $\zz/2\zz$-graded $*$-algebras is given explicitly by
\[
\bra{\xi} \ot (x \hot y) \ot \ket{\eta} \mapsto \inn{\xi, c_H(x) \ga_H^{\Tex{deg}(y)}(\eta)}_{\sE_H} \cd y
\]
for $\xi,\eta \in \sE_H$, $x \in \Cl_H(M)$ and homogeneous $y \in \Cl_V(M)$.
\end{proof}


\begin{prop}
\label{prop:spinors-vert-hor}
We have an isometric isomorphism of $\zz/2\zz$-graded $C^\infty(M)$-modules:
\[
\sE_H \ot_{C^\infty(M)} \sE_V \cong \sE_M
\]
which is compatible with the respective Clifford actions under the isomorphism $\Cl_H(M) \hot_{C^\infty(M)} \Cl_V(M) \cong \Cl(M)$ (graded tensor product). Remark here that $\sE_H \ot_{C^\infty(M)} \sE_V$ is $\zz/2\zz$-graded with grading operator $\ga_H \ot \ga_V$ and that it carries the hermitian form
\[
\inn{ \xi_1 \ot \eta_1, \xi_2 \ot \eta_2} := \inn{\eta_1, \eta_2}_{\sE_V} \cd \inn{\xi_1,\xi_2}_{\sE_H} \q \xi_1,\xi_2 \in \sE_H \, , \, \, \eta_1,\eta_2 \in \sE_V.
\]
Notice also that $\Cl_H(M) \hot_{C^\infty(M)} \Cl_V(M)$ acts on $\sE_H \ot_{C^\infty(M)} \sE_V$ according to the rule
\[
(x \hot y)( \xi \ot \eta) := c_H(x) \ga_H^{\Tex{deg}(y)}(\xi) \ot c_V(y)(\eta)
\]
for $x \in \Cl_H(M)$ and homogeneous $y \in \Cl_V(M)$.
\end{prop}
\proof
By Lemma \ref{l:horiclif} we have the following isomorphism of $C^\infty(M)$-modules:
\[
\begin{split}
\sE_H \ot_{C^\infty(M)} \sE_H^* \ot_{\Cl_H(M)} \sE_M & \cong \Tex{End}_{C^\infty(M)}(\sE_H) \ot_{\Cl_H(M)} \sE_M \\ 
& \cong \Cl_H(M) \ot_{\Cl_H(M)} \sE_M \cong \sE_M.
\end{split}
\]
We leave it to the reader to verify that this isomorphism is compatible with the hermitian forms, the gradings and the Clifford actions. 
\endproof

Our next task is to find an explicit even metric Clifford connection on $\sE_V$ and compare the tensor sum of it with $\Na^{\sE_H}$ with the Clifford connection $\nabla^{\sE_M}$. We let $\Na^{\sE_H^*}$ denote the dual connection, on $\sE_H^*$, described by the formula
\[
\Na_X^{\sE_H^*}( \bra{\xi}) := \bra{ \Na_X^{\sE_H}(\xi) }
\]
for real vector fields $X \in \sX(M)$ and $\xi \in \sE_H$. We remark that the naive choice 
\[
\Na^{\sE_H^*}_X \ot 1 + 1 \ot \Na^{\sE_M}_X : \sE_H^* \ot_{C^\infty(M)} \sE_M \to \sE_H^* \ot_{C^\infty(M)} \sE_M \q X \in \sX(M)
\]
for the connection acting on $\sE_H^* \ot_{\Cl_H(M)} \sE_M$ does not work as this connection is \emph{only} well-defined on $\sE_H^* \ot_{C^\infty(M)} \sE_M$. Instead, we need to introduce correction terms, that make this connection well-defined on $\sE_V$, expressed in terms of the tensor $\omega \in \Om^1(M) \ot_{C^\infty(M)} \Om^2(M)$ of Definition \ref{defn:tensors}. 

We apply the notation $\sharp : \Om^1(M) \to \sX(M)$ and $\flat : \sX(M) \to \Om^1(M)$ for the musical isomorphisms. We then define the homomorphisms
$c : \Om^2(M) \to \Tex{End}_{C^\infty(M)}(\sE_M)$ and  $\wit c : \Om^2(M) \to \Om^1(M) \ot_{C^\infty(M)} \Tex{End}_{C^\infty(M)}(\sE_M)$ by the formulae 
\begin{equation}\label{eq:clifftwo}
\begin{split}
c(\om_1 \we \om_2) & := [ c( \om_1^\sharp), c(\om_2^\sharp)] \\
\wit c(\om_1 \we \om_2) & := \om_1 \ot c(\om_2^\sharp) - \om_2 \ot c(\om_1^\sharp)
\end{split}
\end{equation}
for all $\om_1, \om_2 \in \Om^1(M)$.

\begin{lma}\label{l:commcliff}
We have the relation
\[
\frac{1}{4}[ c(X), c( \om) ] = \wit c(\om)(X)
\]
for all $X \in \sX(M)$ and $\om \in \Om^2(M)$.
\end{lma}
\proof
Without loss of generality we may assume that $\om = \bra{X_1} \we \bra{X_2}$ for real vector fields $X_1$ and $X_2$ on $M$. Furthermore, we may assume that $X \in \sX(M)$ is a real vector field. We thus have that $c( \om) = [c(X_1),c(X_2)]$ and an application of the Clifford relation $c(Z) c(Y) + c(Y) c(Z) = 2 \inn{Z,Y}_M$ (for \emph{real} vector fields $Y,Z \in \sX(M)$) entails that 
\[
\begin{split}
& \frac{1}{4}[ c(X), c( \om) ] = \frac{1}{4} \big[c(X), [c(X_1),c(X_2)]\big] \\
& \q = \inn{X,X_1}_M c(X_2) - \inn{X,X_2}_M c(X_1)
= \wit c(\om_1 \we \om_2)(X) \qedhere
\end{split}
\]
\endproof

%

\begin{prop}
\label{prop:spin-connections}
The following formula defines an even metric Clifford connection $\nabla^{\sE_V}$ on $\sE_V = \sE_H^* \ot_{\Cl_H(M)} \sE_M$:
\begin{align*}
\nabla_X^{\sE_V} (\phi \ot s) &= \Na_X^{\sE_H^*}(\phi) \ot s + \phi \ot \Na^{\sE_M}_X(s) + \phi \ot \frac{1}{4} c(\omega(X))(s) \, , 
\end{align*}
where $X \in \sX(M)$, $\phi \in \sE_H^*$ and $s \in \sE_M$.
\end{prop}
\proof
Let $X \in \sX(M)$ be real and let $\xi \in \sE_H$ and $s \in \sE_M$ be given.

We start by showing that $\nabla_X^{\sE_V}$ is well-defined on $\sE_H^* \ot_{\Cl_H(M)} \sE_M$. Thus, let $Y \in \sX_H(M)$ be real. We need to establish that 
\begin{equation}\label{eq:welldef}
\nabla_X^{\sE_V} ( \bra{\xi} \ot c(Y) s) = \Na_X^{\sE_V}( \bra{c_H(Y)\xi} \ot s).
\end{equation}
We compute as follows:
\[
\begin{split}
& \nabla_X^{\sE_V} ( \bra{\xi} \ot c(Y) s) \\
& \q = \bra{ c_H(Y) \Na_X^{\sE_H} \xi} \ot s + \bra{\xi} \ot \big( \Na^{\sE_M}_X + \frac{1}{4} c(\om(X)) \big) c(Y) s \\
& \q = \Na_X^{\sE_V} ( \bra{c_H(Y) \xi} \ot s) \\
& \qqq + \bra{\xi} \ot \big( c( \Na^M_X(Y) - \Na_X^H(Y)) + \frac{1}{4} [ c(\om(X)), c(Y)] \big)(s)
\end{split}
\]
where the second identity follows since both $\Na^{\sE_M}$ and $\Na^{\sE_H}$ are Clifford connections, see \eqref{eq:clicon} and \eqref{eq:horclicon}. The identity in \eqref{eq:welldef} is now a consequence of Proposition \ref{p:omedir} and Lemma \ref{l:commcliff}.

It is clear that $\Na_X^{\sE_V}$ commutes with the grading operator $\ga_V = \ga_H^* \ot \ga_M$ on $\sE_V$.
\medskip

We continue by showing that $\nabla^{\sE_V}$ is a Clifford connection for the action of $\Cl_V(M)$. This amounts to proving the identity
\[
[ \nabla_X^{\sE_V}, c_V(Z) ] (\bra{\xi} \ot s)= c_V(\nabla^V_X(Z)) (\bra{\xi} \ot s)
\]
for any vertical vector field $Z$. We thus compute as follows:
\[
\begin{split}
& \nabla_X^{\sE_V} c_V(Z) (\bra{\xi} \ot s) \\
& \q = \bra{ \Na^{\sE_H}_X \ga_H \xi} \ot c(Z) s  
+ \bra{\ga_H \xi} \ot \big( \Na_X^{\sE_M} + \frac{1}{4} c(\om(X)) \big) c(Z) s \\
& \q = c_V(Z) \Na_X^{\sE_V}( \bra{\xi} \ot s)
+ \bra{\ga_H \xi} \ot \big( c( \Na_X^M(Z)) + \frac{1}{4} [ c(\om(X)), c(Z)] \big) s
\end{split}
\]
It therefore suffices to verify that
\[
c( \Na_X^M(Z)) - \frac{1}{4} [c(Z),c( \om(X))] = c( \Na^V_X(Z)).
\]
But this is a consequence of Proposition \ref{p:omedir} and Lemma \ref{l:commcliff}.
\medskip

We finally show that $\Na^{\sE_V}$ is metric. Thus, let $\eta \in \sE_H$ and $t \in \sE_M$ be two extra elements. By Lemma \ref{l:horiclif} we may find a unique $y \in \Cl_H(M)$ with $c_H(y) = \ket{\xi}\bra{\eta}$. Using that both $\Na^{\sE_M}$ and $\Na^{\sE_H}$ are metric Clifford connections we then compute that
\[
\begin{split}
& \binn{\Na_X^{\sE_V}(\bra{\xi} \ot s), \bra{\eta} \ot t}_{\sE_V} + \binn{\bra{\xi} \ot s, \Na_X^{\sE_V}(\bra{\eta} \ot t)}_{\sE_V} \\
& \q = X\big( \inn{\bra{\xi} \ot s, \bra{\eta} \ot t}_{\sE_V} \big) + \inn{s, c( \Na_X^H(y) - \Na_X^M(y))(t)}_{\sE_M} \\
& \qqq - \frac{1}{4}\binn{s, [ c( \om(X)), c(y) ](t)}_{\sE_M}.
\end{split}
\]
The fact that $\Na^{\sE_V}$ is metric can now be derived from Proposition \ref{p:omedir} and Lemma \ref{l:commcliff}.
\endproof


The following result is a restatement of \cite[Lemma 10.13]{BGV92}. The {\em mean curvature} of the Riemannian submersion $\pi : M \to B$ is defined as the $1$-form
\begin{equation}\label{eq:mean-curv}
k := (\Tex{Tr} \ot 1)(S) \in \Om^1(M)
\end{equation}
where $\Tex{Tr} : \Om^1(M) \ot_{C^\infty(M)} \Om^1(M) \to C^\infty(M)$ is the {\em trace} defined by $\Tex{Tr}( \bra{X} \ot \bra{Y} ) := \inn{Y,X}_M$ for real vector fields $X,Y \in \sX(M)$.

\begin{lma}\label{lma:mean-curv}
We have the identity
\[
(c \ot c)(\sharp \ot 1)(\omega) = - 2 c(k^\sharp) - \frac{1}{2}(c \ot c)(1 \ot \sharp)(\Omega) .
\]
\end{lma} 
\proof
To ease the notation we introduce the elements $\wit S$ and $\wit \Om$ in $\Om^1(M) \ot_{C^\infty(M)} \Om^2(M)$ defined by
\[
\begin{split}
& \wit S(X,Y,Z) := S(X,Y,Z) - S(X,Z,Y) \q \Tex{and} \\
& \wit \Om(X,Y,Z) := \Om(X,Z,Y) - \Om(X,Y,Z) + \Om(Y,Z,X) 
\end{split}
\]
and we thus have that
\[
\om(X,Y,Z) := \om(X)(Y,Z) = - \wit S(X,Y,Z) + \half \wit \Om(X,Y,Z).
\]

We claim that 
\begin{equation}\label{eq:traome}
\begin{split}
& (c \ot c)(\sharp \ot 1)( \wit S) = 2 \big(\Tex{Tr} \ot (c \ci \sharp) \big)(S), \q \Tex{and} \\ 
& (c \ot c)(\sharp \ot 1)(\wit \Om) = - (c \ot c)(1 \ot \sharp)(\Om).
\end{split}
\end{equation}

To prove the first of these identities we write
\[
S = \sum_{i,j,k} \bra{X_i^1} \ot \bra{X_j^2} \ot \bra{X_k^3}
\]
for real vector fields $X_i^1, X_j^2 \in \sX_V(M)$ and $X_k^3 \in \sX_H(M)$, and we notice that the symmetry relation in Proposition \ref{prop:expression-tensor} implies that
\[
S = \half \sum_{i,j,k} \big( \bra{X_i^1} \ot \bra{X_j^2} \ot \bra{X_k^3} + \bra{X_j^2} \ot \bra{X_i^1} \ot \bra{X_k^3} \big).
\]
Using the Clifford relation we then obtain that
\[
\begin{split}
& (c \ot c)(\sharp \ot 1)( \wit S) = \frac{1}{2} \sum_{i,j,k} \big( c(X_i^1) [ c( X_j^2), c(X_k^3)]
+ c(X_j^2) [c(X_i^1), c(X_k^3)] \big) \\
& \q = 2 \sum_{i,j,k} \inn{X_i^1,X_j^2}_M \cd c(X_k^3)
= 2 \big( \Tex{Tr} \ot (c \ci \sharp) \big)(S) \, ,
\end{split}
\]
proving the first identity in \eqref{eq:traome}.

To verify the second identity in \eqref{eq:traome} we write
\[
\Om = \sum_{i,j,k} \bra{Y_i^1} \we \bra{Y_j^2} \ot \bra{Y_k^3}
\]
for real vector fields $Y_i^1, Y_j^2 \in \sX_H(M)$ and $Y_k^3 \in \sX_V(M)$. We then have that
\[
\wit \Om = \sum_{i,j,k} \Big( \bra{Y_j^2} \ot \bra{Y_i^1} \we \bra{Y_k^3} + \bra{Y_k^3} \ot \bra{Y_i^1} \we \bra{Y_j^2} - \bra{Y_i^1} \ot \bra{Y_j^2} \we \bra{Y_k^3} \Big) .
\]
Using the Clifford relation we obtain that
\[
\begin{split}
(c \ot c)(\sharp \ot 1)( \wit \Om) & = \sum_{i,j,k} \Big( 
c(Y_j^2) [c(Y_i^1),c(Y_k^3)] + c(Y_k^3) [ c(Y_i^1), c(Y_j^2)] \\
& \qqq - c(Y_i^1) [ c(Y_j^2),  c(Y_k^3) ] \Big) \\
& = - \sum_{i,j,k} [ c(Y_i^1), c(Y_j^2) ] c(Y_k^3) 
= - (c \ot c)(1 \ot \sharp)(\Om) \, , 
\end{split}
\]
and this ends the proof of the lemma.
\endproof

\section{Factorization in unbounded $KK$-theory}
\label{sect:fact}
In this section we come to the main result of this paper, which is the factorization in unbounded $KK$-theory of the Dirac operator $D_M$ on $M$ in terms of a vertical operator $S$ and the Dirac operator $D_B$ on $B$ for a Riemannian submersion $M \to B$ of compact spin$^{\Tex{c}}$ manifolds. We let $L^2(\sE_M)$ and $L^2(\sE_B)$ denote the Hilbert space completions of the spinor modules $\sE_M$ and $\sE_B$ with respect to the hermitian forms and the measures $\mu_M$ and $\mu_B$ coming from the Riemannian metrics. Our task is then to find a $C^*$-correspondence from $C(M)$ to $C(B)$, together with a self-adjoint and regular unbounded operator $S$ on $X$, such that 
\[
L^2(\sE_M) \cong X \hot_{C(B)} L^2(\sE_B)
\]
and in such a way that the operator $D_M$ corresponds to the tensor sum $S \otimes \ga_B + 1 \otimes_\nabla D_B$ for some metric connection $\nabla$ on $X$. 

First, let us translate the results of the previous section on connections on the vertical spinor bundle on $M$ to a Hilbert $C^*$-module over $C(B)$. In fact, since $\sE_V$ is a $C^\infty(M)$-module, it becomes a $C^\infty(B)$-module (denoted $\cc X$) using the pullback map $\pi^* : C^\infty(B) \to C^\infty(M)$. Moreover, $\cc X$ can be equipped with a $C^\infty(B)$-valued inner product $\langle \cdot,\cdot\rangle$ defined by integration along the fibers:
\begin{align*}
\langle \phi_1,\phi_2 \rangle_{\cc X} (b) := \int_{F_b} \langle \phi_1,\phi_2 \rangle_{\sE_V} \, d\mu_{F_b} 
\end{align*}
using the measures coming from the Riemannian metrics on the Riemannian submanifolds of $M$: $F_b := \pi^{-1}(\{b\})$, $b \in B$. We define the Hilbert $C^*$-module $X$ to be the completion of $\cc X$ in the norm coming from this inner product and the $C^*$-norm on $C(B)$. There is a left action of $C(M)$ on $X$ (coming from the left action of $C^\infty(M)$ on $\sE_V$) and this gives $X$ the structure of a $C^*$-correspondence from $C(M)$ to $C(B)$. Moreover, $X$ is $\zz/2\zz$-graded with grading operator $\ga_X$ induced by the grading operator $\ga_V$ on $\sE_V$.
\medskip

In order to relate the interior tensor product $X \hot_{C(B)} L^2(\sE_B)$ to $L^2(\sE_M)$ we fix coordinate charts $(V,\phi)$ on $B$ and $(W,\psi)$ on a model fiber $F$ together with a local trivialization
\[
\rho : \pi^{-1}(V) \to V \times F.
\]
Putting $U := \rho^{-1}( V \times W)$ we have a coordinate chart $(U,\si)$ on $M$ with
\[
\si := (\phi \ti \psi) \ci \rho : U \to \rr^{\Tex{dim}(B) + \Tex{dim}(F)} .
\]
We refer to such a chart as a \emph{fibration chart}. Combining the fibration chart $(U,\si)$ with the Riemannian metric we obtain the positive invertible matrix of smooth functions
\[
g : U \to \Tex{GL}_{\Tex{dim}(M)}(\rr)_+ \qq g_{ij} := \inn{ \pa/\pa \si_i, \pa/\pa \si_j}_M.
\]
Furthermore, letting $Q : \rr^{\Tex{dim}(M)} \to \rr^{\Tex{dim}(M)}$ denote the projection $Q : (t_1,\ldots,t_{\Tex{dim}(M)}) \mapsto (0,\ldots,0,t_{\Tex{dim}(B) + 1},\ldots,t_{\Tex{dim}(M)})$ onto the last $\Tex{dim}(F)$ copies of $\rr$, we have the following invertible positive matrix of smooth functions:
\[
Q g Q : U \to \Tex{GL}_{\Tex{dim}(F)}(\rr)_+.
\]
We may also combine the coordinate chart $(V,\phi)$ with the Riemannian metric on the base obtaining the invertible positive matrix of smooth functions:
\[
h : V \to \Tex{GL}_{\Tex{dim}(B)}(\rr)_+.
\]
By taking determinants, we thus obtain the three positive functions
\[
\Tex{det}(g) \, , \, \, \Tex{det}( h) \circ \pi \, \Tex{ and } \, \, \Tex{det}( Q g Q ) : U \to (0,\infty).
\]
%

\begin{lma}\label{l:indepchart}
We have the identity
\[
\Tex{det}(g) = \big( \Tex{det}(h) \ci \pi \big) \cdot \Tex{det}( Q g Q).
\]
\end{lma}
\begin{proof}
We define the upper triangular matrix of smooth functions $O : U \to \Tex{GL}_{\Tex{dim}(M)}(\rr)$ by putting 
\[
O_{ij} := \fork{cc}{
1 & i = j \\
- d\si_j( P(\pa/\pa \si_i)) & i \leq \Tex{dim}(B) < j \\
0 & \Tex{elsewhere}
} .
\]
For $i \in \{1,\ldots, \Tex{dim}(M)\}$ we compute that
\[
\sum_{j = 1}^{\Tex{dim}(M)} O_{ij} \cd \pa/\pa \si_j
= \fork{cc}{
\pa/\pa \si_i &  i > \Tex{dim}(B) \\
(1 - P)(\pa/\pa \si_i) & i \leq \Tex{dim}(B)
} \, , 
\]
and we thus have that
\[
(O g O^t)_{kl} = 
\fork{ccc}{ g_{kl} & k,l > \Tex{dim}(B) \\
\inn{ (1 - P) \pa/\pa \si_k, (1 - P) \pa/\pa \si_l} & k,l \leq \Tex{dim}(B) \\
0 & \Tex{elsewhere}
} \, . 
\]
Using that $\pi : M \to B$ is a Riemannian submersion we conclude that $O g O^t = (h \ci \pi) \op Q g Q$. Therefore, since $O$ is upper-triangular and $1$ on the diagonal, we obtain that
\[
\Tex{det}(g) = \Tex{det}( O g O^t) = \big( \Tex{det}( h) \ci \pi \big) \cd \Tex{det}( Q g Q) \, , 
\]
proving the lemma.
\end{proof}



\begin{prop}
The left $C^\infty(M)$-module map 
\begin{align*}
V : \cc X \otimes_{C^\infty(B)} \sE_B &\to \sE_M \\
V : ( \bra{\xi} \ot s) \ot r &\mapsto \big( \ket{r \ot 1} \bra{\xi} \big)(s)
\end{align*}
defined for $\xi \in \sE_H$, $s \in \sE_M$ and $r \in \sE_B$, induces a unitary isomorphism 
\[
V : X \hot_{C(B)} L^2(\sE_B) \to L^2(\sE_M)
\]
of $\zz/2\zz$-graded Hilbert spaces.
\end{prop}
\proof
Using that $\Tex{End}_{C^\infty(M)}(\sE_H) \cong \Cl_H(M)$ (see Lemma \ref{l:horiclif}) and the fact that $\sE_B$ (and hence also $\sE_H$) are finitely generated projective modules, it follows that $V : \cc X \ot_{C^\infty(B)} \sE_B \to \sE_M$ is an isomorphism of $\zz/2\zz$-graded left modules over $C^\infty(M)$. By a density argument we obtain the result of the proposition if we can establish that
\begin{equation}\label{eq:isomvee}
\begin{split}
& \binn{ V( \bra{\xi_1} \ot s_1 \ot r_1), V( \bra{\xi_2} \ot s_2 \ot r_2) }_{L^2(\sE_M)}  \\
& \q = \binn{ \bra{\xi_1} \ot s_1 \ot r_1, \bra{\xi_2} \ot s_2 \ot r_2 }_{X \hot_{C(B)} L^2(\sE_B)} 
\end{split}
\end{equation}
for all $\xi_1,\xi_2 \in \sE_H$, $s_1,s_2 \in \sE_M$ and $r_1,r_2 \in \sE_B$. To ease the notation we put
\[
T := \ket{\xi_1}\bra{\xi_2} \in \Tex{End}_{C^{\infty}(M)}(\sE_H) \su \Tex{End}_{C^\infty(M)}( \sE_M) \, , 
\]
where we suppress the identifications $\Tex{End}_{C^{\infty}(M)}(\sE_H) \cong \Cl_H(M) \su \Tex{End}_{C^\infty(M)}( \sE_M)$. Moreover, without loss of generality, we assume that $\Tex{supp}(s_1) \su U$. 

We expand on the right-hand side of \eqref{eq:isomvee}: 
\[
\begin{split}
& \binn{ \bra{\xi_1} \ot s_1 \ot r_1, \bra{\xi_2} \ot s_2 \ot r_2 }_{X \hot_{C(B)} L^2(\sE_B)} \\
& \q = \int_B \inn{r_1,r_2}_{\sE_B} \cd \binn{\bra{\xi_1} \ot s_1, \bra{\xi_2} \ot s_2)}_{\sX} \, d\mu_B  .
\end{split}
\]
Now, for each $b \in \pi(U)$ we compute that
\[
\begin{split}
& \inn{\bra{\xi_1} \ot s_1, \bra{\xi_2} \ot s_2}_{\sX}(b) \\
& \q = \int_{\psi(W)}\Big( \inn{s_1, T(s_2)}_{\sE_M} \cd \Tex{det}(QgQ)^{1/2} \Big)
\big( \si^{-1}(\phi(b), x ) \big) \, dm_{\Tex{dim}(F)}(x) \, , 
\end{split}
\]
where $m_{\Tex{dim}(F)}$ is Lebesgue measure on $\rr^{\dim(F)}$. Moreover, $\inn{s_1, T(s_2)}(b) = 0$ for all $b \notin \pi(U)$. Using Lemma \ref{l:indepchart} we thus have that 
\[
\begin{split}
& \binn{ \bra{\xi_1} \ot s_1 \ot r_1, \bra{\xi_2} \ot s_2 \ot r_2 }_{X \hot_{C(B)} L^2(\sE_B)} \\
& \q = \int_{\phi(V) \ti \psi(W)} 
\Big( \big( \inn{ r_1, r_2 }_{\sE_B} \cd \Tex{det}(h)^{1/2} \big) \ci \pi \Big)\big( \si^{-1}(x) \big) \\
& \qqq \cd \Big(  \inn{s_1,T(s_2)}_{\sE_M} \cd \Tex{det}( Qg Q)^{1/2} \Big)\big( \si^{-1}(x) \big) 
\, dm_{\Tex{dim}(M)}(x)  \\
& \q = \int_M ( \inn{r_1,r_2}_{\sE_B} \ci \pi) \cd \inn{s_1,T(s_2)}_{\sE_M} \, d \mu_M \, , 
\end{split}
\]
and the result of the proposition can be obtained by noting that
\[
\begin{split}
& \int_M (\inn{r_1,r_2}_{\sE_B} \ci \pi)\cd \inn{s_1,T(s_2)}_{\sE_M} \, d \mu_M \\
& \q = 
\binn{ V( \bra{\xi_1} \ot s_1 \ot r_1), V(\bra{\xi_2} \ot s_2 \ot r_2) }_{L^2(\sE_M)} \, .\qedhere
\end{split}
\]
\endproof

We will now use the connection $\nabla^{\sE_V}$ constructed in Proposition \ref{prop:spin-connections} to define an unbounded operator on $X$. 

\begin{lma}\label{l:symmetry}
The following local expression defines an odd symmetric unbounded operator $S_0 : \cc X \to X$:
\[
S_0(\xi) = i  \sum_{j=1}^{\dim (F)} c_V(e_j) \nabla^{\sE_V}_{e_j}(\xi) 
\]
where $\{ e_j\}$ is a local orthonormal frame for $\sX_V(M)$ consisting of real vector fields. 
\end{lma}
\proof
We show that $S_0$ is symmetric, thus that
\begin{equation}\label{eq:symmetry}
\inn{S_0(\phi_1),\phi_2}_{\cc X} = \inn{\phi_1, S_0(\phi_2)}_{\cc X} \q \phi_1,\phi_2 \in \cc X \, .
\end{equation}
Without loss of generality, assume that $\Tex{supp}(\phi_1), \Tex{supp}(\phi_2) \su U$ where $(U, \si)$ is a fibration chart. We then notice that
\[
S_0(\phi_1) = i  \sum_{k,l > \dim(B)} c_V\big( (Q g Q)^{-1}_{kl} \pa/\pa \si_l \big) \Na^{\sE_V}_{\pa/ \pa \si_l}(\phi_1) \, , 
\]
and similarly for $\phi_2$. To ease the notation, we define the vertical vector fields
\[
X_l := \sum_{k > \dim(B)} (Q g Q)^{-1}_{kl} \pa/\pa \si_l \q l \in \{ \dim(B) + 1,\ldots, \dim(M) \} \, .
\]
Let $b \in \pi(U)$. Using that $\Na^{\sE_V}$ is a metric Clifford connection (see Proposition \ref{prop:spin-connections}) we obtain that
\[
\begin{split}
& \inn{S_0(\phi_1),\phi_2}_{\cc X}(b) 
= - i \sum_{l > \dim(B)}\int_{F_b} \inn{ c_V(X_l) \Na^{\sE_V}_{\pa/\pa \si_l}(\phi_1), \phi_2}_{\sE_V} \, d \mu_{F_b} \\
& \q = \inn{\phi_1, S_0(\phi_2)}_{\cc X}(b)
+ i \sum_{l > \dim(B)} \int_{F_b} \inn{\phi_1, c_V( \Na^V_{\pa/\pa \si_l}(X_l))(\phi_2)}_{\sE_V} \, d \mu_{F_b} \\
& \qqq - i \sum_{l > \dim(B)} \int_{F_b} \pa/\pa \si_l\big( \inn{\phi_1, c_V(X_l)(\phi_2)}_{\sE_V} \big) \, d \mu_{F_b}.
\end{split}
\]
Using the det/log relationship 
\[
\pa/\pa \si_l( \Tex{det}(QgQ)^{-1/2}) = -\frac{1}{2}\Tex{Tr}\big( \pa/\pa \si_l( QgQ) \cd (QgQ)^{-1} \big) \cd \Tex{det}(QgQ)^{-1/2}
\]
we then reduce the proof of \eqref{eq:symmetry} to a verification of the identity
\begin{equation}\label{eq:conntrac}
\sum_{l > \dim(B)} \Na^V_{\pa/\pa \si_l}(X_l) = \frac{1}{2}\sum_{l > \dim(B)} \Tex{Tr}\big( \pa/\pa \si_l( QgQ) \cd (QgQ)^{-1} \big) \cd X_l
\end{equation}
of vertical vector fields. However, using Koszul's formula for the Levi-Civita connection on $M$, we see that
\[
\begin{split}
& \sum_{l > \dim(B)} \Na^V_{\pa/\pa \si_l}(X_l)
= \sum_{j,l > \dim(B)} \inn{ \Na^M_{\pa/\pa \si_l}(X_l), X_j}_M \cd \pa/\pa \si_j \\
& \q = \frac{1}{2}\sum_{j,l > \dim(B)} \big( \pa/\pa \si_l( \inn{X_l,X_j}_M ) - \inn{ [X_l,X_j], \pa/\pa \si_l}_M \big) \cd \pa/\pa \si_j \\
& \q = \frac{1}{2}\sum_{l > \dim(B)} \Tex{Tr}\big( \pa/\pa \si_l( QgQ) \cd (QgQ)^{-1} \big) \cd X_l \, ,
\end{split}
\]
thus proving \eqref{eq:conntrac} and thereby also that $S_0$ is symmetric. 
%
\endproof

We let $S : \dom(S) \to X$ denote the closure of $S_0 : \cc X \to X$. The vertical part of our geometric data is then encoded in the following:

\begin{prop}\label{p:unboukas}
The triple $(C^\infty(M),X,S)$ is an even unbounded Kasparov module from $C(M)$ to $C(B)$ with grading operator $\ga_X : X \to X$.
\end{prop}
\proof
Let $\si : \Om^1(M) \to \Tex{End}_{C^\infty(M)}(\sE_V)$ denote the symbol of the first order differential operator $S_0 : \sE_V \to \sE_V$. By Lemma \ref{l:symmetry} and \cite{H16,HK16}, it suffices to verify that $\si( \om(x))$ is invertible whenever $x \in M$, $\om \in \Om^1(M)$ and $\om(x) : (T_V M)_x \to \C$ is non-trivial (where $T_V M \to M$ denotes the vertical tangent bundle). But this follows from the local formula $\si(df) = [S_0, f] = i \sum_{j = 1}^{\dim(F)} c_V(e_j) e_j(f)$ which holds for all $f \in C^\infty(M)$. Indeed, this formula implies that $\si( \om) = c_V( (\om \ci P)^\sharp)$ for all $\om \in \Om^1(M)$.
\endproof


Next, consider the Dirac operator $D_B : \dom(D_B) \to L^2(\sE_B)$. It is defined as the closure of the unbounded operator
\[
(D_B)_0 = i  \sum_{\al = 1}^{\dim (B)} c( f_\alpha) \nabla^{\sE_B}_{f_\alpha} : \sE_B \to L^2(\sE_B)
\]
defined locally for any local orthonormal frame $\{ f_\al\}$ for $\sX(B)$ consisting of real vector fields.

The horizontal part of our geometric data can then be expressed by saying that $(C^\infty(B),L^2(\sE_B),D_B)$ is an even spectral triple with grading operator $\ga_B : L^2(\sE_B) \to L^2(\sE_B)$.

In order to form the unbounded Kasparov product of the vertical and the horizontal components we need to lift the Dirac operator $D_B$ to an unbounded selfadjoint operator on the Hilbert space $X \hot_{C(B)} L^2(\sE_B)$. To carry this out, we need a metric connection on the Hilbert $C^*$-module $X$.

We let $\Om^1_{\Tex{cont}}(B)$ denote the $C^*$-correspondence from $C(B)$ to $C(B)$ defined as the completion of the smooth form $\Om^1(B)$ with respect to (the norm coming from) the $C(B)$-valued inner product $\inn{\om_1,\om_2} := \inn{\om_1^\sharp,\om_2^\sharp}_B$.

\begin{defn}
The metric connection $\nabla^{\cc X} : \cc X \to X \hot_{C(B)} \Omega^1_{\Tex{cont}}(B)$ is defined for $Z \in \sX(B)$ by
\[
\nabla_Z^{\cc X} (\xi) = \nabla_{Z_H}^{\sE_V} (\xi) + \frac{1}{2} k(Z_H) \cdot \xi
\]
with $k(Z_H) \in C^\infty(M)$ the mean curvature from \eqref{eq:mean-curv}.
\end{defn}

We need to verify that the linear map $\nabla^{\cc X}$ is indeed a metric connection. To this end, we first establish a local formula for the mean curvature:

\begin{lma}\label{l:locmean}
Let $(V,\phi)$ and $(W,\psi)$ be coordinate charts on $B$ and $F$, resp. and suppose that $\rho : \pi^{-1}(V) \to V \ti F$ is a local trivialization. For any vector field $Z \in \sX(B)$, we have the local formula:
\[
\begin{split}
k( Z_H) 
& := Z_H( \Tex{det}^{1/2}( Qg Q) ) \cdot \Tex{det}^{-1/2}( Qg Q)
+ \sum_{i > \dim(B)}^{\dim(M)} \pa/\pa \si_i \big( d \si_i( Z_H) \big) \, , 
\end{split}
\]
where $(U,\si)$ is the coordinate chart given by $U := \rho^{-1}(V \ti W)$, $\si := (\phi \ti \psi) \ci \rho$.
\end{lma}
\begin{proof}
Without loss of generality, we assume that $Z \in \sX(B)$ is real. Recall that
\[
k(Z_H) = \Tex{Tr}(S (Z_H) ) \, , 
\]
where $S(Z_H) \in \Tex{End}_{C^\infty(M)}( \sX(M/B))$ is given by
\[
\inn{ X, S(Z_H)(Y) }_M = \frac{1}{2} \big( Z_H( \inn{X,Y}_M) - \inn{ [Z_H,X],Y}_M - \inn{ X, [Z_H,Y]}_M \big) \, .
\]
We have the local formula
\[
\begin{split}
k(Z_H)
& = \sum_{i > \dim(B)}^{\Tex{dim}(M)} \binn{ (d \si_i \ci P)^\sharp, S(Z_H)(\pa/\pa \si_i) }_M \\
& = \sum_{i,j > \dim(B)}^{\Tex{dim}(M)} \big( (Q g Q)^{-1} \big)_{ij} \cd \binn{ \pa/\pa \si_j, S(Z_H)(\pa/\pa \si_i) }_M \, .
\end{split}
\]
Using the explicit formula for $S(Z_H)$ a direct computation then shows that
\[
k(Z_H) = \frac{1}{2} \Tex{Tr}\big( (Q g Q)^{-1} Z_H( Qg Q ) \big)
+ \sum_{i > \dim(B)}^{\Tex{dim}(M)} \pa/\pa \si_i ( d \si_i (Z_H)) \, .
\]
%
The result of the lemma thus follows from the det/log relationship:
\[
Z_H( \Tex{det}^{1/2}( Qg Q)) = \frac{1}{2}\Tex{Tr}( (Q g Q)^{-1} Z_H( Q g Q) ) \cd \Tex{det}^{1/2}(Q g Q). \qedhere
\]
\end{proof}

\begin{prop}\label{p:metconX}
Let $Z \in \sX(B)$, $\xi, \eta \in \cc X$ and $f \in C^\infty(B)$. We have the identities
\[
\Na^{\cc X}_{Z \cd f}(\xi) = \Na^{\cc X}_Z(\xi) \cd f \q 
\Na^{\cc X}_Z(\xi \cd f) = \Na^{\cc X}_Z(\xi) \cd f + \xi \cd Z(f)
\]
as well as the identity
\[
\inn{ \Na^{\cc X}_Z( \xi),\eta } + \inn{\xi,\Na^{\cc X}_Z(\eta)} = Z( \inn{\xi,\eta})
\]
when $Z$ is real.
\end{prop}
\begin{proof}
The first two identities can be verified by a straightforward computation, so we focus on the third identity. 
%

Without loss of generality we assume that $\Tex{supp}(\xi) \su U$ where $(U, \si)$ is a fibration chart of the form $U = \rho^{-1}(V \ti W)$ and $\si = (\phi \ti \psi)\ci \rho : U \to \rr^{\Tex{dim}(B) + \Tex{dim}(F)}$. We may also assume that $Z = \pa/\pa \phi_i$ for some $i \in \{1,\ldots,\dim(B)\}$. Using that $\Na^{\sE_V}$ is metric, by Proposition \ref{prop:spin-connections}, we then have that
\[
\begin{split}
& \inn{\Na^{\cc X}_Z(\xi),\eta}_X + \inn{\xi,\Na^{\cc X}_Z(\eta)}_X \\
& \q = \int_{\psi(W)} \Big( \big( Z_H (\inn{ \xi,\eta}_{\sE_V}) + \inn{\xi,\eta}_{\sE_V} \cdot k(Z_H) \big) \cdot \Tex{det}^{1/2}( Qg Q) \Big) \\
& \qqq \ci \si^{-1}( \phi(\cd),y) \, dy \, . 
\end{split}
\]
To ease the notation, we put $h := \inn{\xi,\eta}_{\sE_V} \cdot \Tex{det}^{1/2}( Qg Q) : M \to \C$. Using Lemma \ref{l:locmean} we compute the integrand in the above expression:
\[
\begin{split}
& \big( Z_H (\inn{ \xi,\eta}_{\sE_V}) + \inn{\xi,\eta}_{\sE_V} \cdot k(Z_H) \big) \cdot \Tex{det}^{1/2}( Qg Q) \\
& \q = Z_H( h ) + \sum_{k > \dim(B)}^{\dim(M)} \pa/\pa \si_k\big( d \si_k( Z_H) \big) \cd h \\
& \q = \sum_{k > \dim(B)}^{\Tex{dim}(M)} \pa/\pa \si_k\big( d \si_k( Z_H) \cd h \big)
+ \pa/\pa \si_i( h ) \, .
\end{split}
\]
Combining these computations we thus obtain that
\[
\begin{split}
& \inn{\Na^{\cc X}_{\pa/\pa \phi_i}(\xi),\eta}_X + \inn{\xi,\Na^{\cc X}_{\pa/\pa \phi_i}(\eta)}_X \\
& \q = \int_{\psi(W)} \pa/\pa \si_i (h ) \ci \si^{-1}( \phi(\cd),y) \, dy = \pa/\pa \phi_i ( \inn{ \xi,\eta}_X ) \, .
\end{split}
\]
This ends the proof of the proposition.
\end{proof}

We are now ready to define our lift of the Dirac operator $D_B : \dom(D_B) \to L^2(\sE_B)$:

\begin{lma}
The following local expression defines an odd symmetric unbounded operator on the image of $\cc X \otimes_{C^\infty(B)} \sE_B$ in $X \hot_{C(B)} L^2(\sE_B)$:
\[
(1 \otimes_\nabla D_B) (\xi \otimes r) := \xi \otimes D_B r + i \sum_\alpha \nabla^{\cc X}_{f_\alpha} (\xi)  \otimes c(f_\alpha) r  .
\]
\end{lma}
\begin{proof}
The result follows since $\Na^{\cc X}$ is a metric connection, by Proposition \ref{p:metconX}, and since $[D_B,h] = i \sum_{\al = 1}^{\dim(B)} c(f_\al) f_\al(h)$, whenever $h : B \to \C$ has support inside the domain of a local orthonormal frame $\{f_\al\}_{\al = 1}^{\dim(B)}$ of vector fields on $B$.
\end{proof}

The tensor sum we are after is given by the odd symmetric unbounded operator
\[
(S \ti_{\Na} D_B)_0 := S \otimes \ga_B + 1 \otimes_\nabla D_B: 
\dom( S \ti_\Na D_B)_0 \to X \hot_{C(B)} L^2(\sE_B) \, ,
\]
where the domain is the image of $\cc X \otimes_{C^\infty(B)} \sE_B$ in $X \hot_{C(B)} L^2(\sE_B)$. The closure of the symmetric unbounded operator $(S \ti_{\Na} D_B)_0$ will be denoted by $S \times_\nabla D_B$.
%

Before we start comparing $S \times_\Na D_B$ with the Dirac operator $D_M$ we present a useful lemma on the various connections involved. We recall that $\om(X) \in \Om^2(M)$ was introduced in Definition \ref{defn:tensors} and that the maps $c : \Om^2(M) \to \Tex{End}_{C^\infty(M)}(\sE_M)$ and $\wit c : \Om^2(M) \to \Om^1(M) \ot_{C^\infty(M)} \Tex{End}_{C^\infty(M)}(\sE_M)$ are defined in \eqref{eq:clifftwo}.

\begin{lma}\label{l:commconn}
Let $T \in \Tex{End}_{C^\infty(M)}(\sE_H)$ be given. For all $X \in \sX(M)$ we then have the identities
\[
\begin{split}
& [ \Na^{\sE_M}_X, T] 
= [\Na_X^{\sE_H},T] + \frac{1}{4} [ T, c( \omega(X))]%
\end{split}
\]
of endomorphisms of $\sE_M$, where we are suppressing the identifications 
\[
\Tex{End}_{C^\infty(M)}(\sE_H) \cong \Tex{Cl}_H(M) \su \Tex{End}_{C^\infty(M)}( \sE_M) \, .
\]
\end{lma}
\begin{proof}
Without loss of generality, we may assume that $T = c_H(Y_H)$ for some vector field $Y \in \sX(B)$. The fact that $\nabla^{\sE_M}$ and $\Na^{\sE_H}$ are Clifford connections implies that we can use Proposition \ref{p:omedir} for the corresponding connections on vector fields. Indeed: 
\begin{align*}
&[ \Na^{\sE_M}_{X}, c_H(Y_H)] = c(\nabla^H_X(Y_H)) + c( \omega(X)(Y_H,\cdot)^\sharp )\\
&\qquad = [ \Na_X^{\sE_H}, c_H(Y_H) ] + \tilde c (\omega(X))(Y_H)\\
&\qquad = [\Na_X^{\sE_H}, c_H(Y_H) ] + \frac{1}{4} [c(Y_H), c(\omega(X))] \, ,
\end{align*}
where we also used Lemma \ref{l:commcliff} in passing to the last line. 
\end{proof}

We recall that the curvature of our Riemannian submersion is the element $\Om \in \Om^2(M) \ot_{C^\infty(M)} \Om^1(M)$ introduced in Definition \ref{defn:tensors}.
 
\begin{thm}
Under the unitary isomorphism $V: X \hot_{C(B)} L^2(\sE_B) \to L^2(\sE_M)$ we have the identity
\[
V ( S \times_\nabla D_B )V^* = D_M - \frac{i}{8} (c \ot c)(1 \ot \sharp)(\Omega).
\]
\end{thm}
\proof
Let us fix an element $\xi \otimes r \in \cc X \otimes_{C^\infty(B)} \sE_B$ in the core of $S \ti_{\Na} D_B$, with $\xi \in \cc X$ of the form $\xi = \phi \ot s$. Thus, $\phi \in \sE_H^*$ and $s \in \sE_M$. We will show that 
\[
V (S \times_\nabla D_B)(\xi \ot r) = \big( D_M - \frac{i}{8} (c \ot c)(1 \ot \sharp)(\Omega) \big)V(\xi \ot r).
\]
Since $V$ descends to an isomorphism of the core $\cc X \otimes_{C^\infty(B)} \sE_B$ for $S \ti_{\Na} D_B$ with the core $\sE_M$ for $D_M$ the above identity will prove the result of the theorem. Without loss of generality, we may assume that $\Tex{supp}(\xi) \subset U$ and $\Tex{supp}(r) \su \pi(U)$ for some open set $U \su M$ admitting real local orthonormal frames $\{ e_j\}_{j = 1}^{\dim (F)}$ and $\{ f_\al\}_{\al = 1}^{\dim (B)}$ defined on $U$ and $\pi(U)$, resp. 

We compute the vertical and the horizontal part of $V ( S \times_\nabla D_B)( \xi \ot r)$ separately. To ease the notation, we put
\[
T := \ket{r \ot 1}\phi \in \Tex{End}_{C^\infty(M)}(\sE_H)
\]
and notice that $V(\xi \ot r) = T(s)$.

We first remark that it follows by Lemma \ref{l:commconn} that
\begin{align}
\label{eq:vert}
& V( \Na_X^{\sE_V}(\xi) \ot r) \\\nonumber 
& \qquad 
 = \big( \ket{r \ot 1}\Na_X^{\sE_H^*}(\phi) \big)(s)
+ T \Na_X^{\sE_M}(s) + \frac{1}{4} T c( \om(X))(s) \\\nonumber 
& \qquad = \Na_X^{\sE_M} T(s) - \big( \ket{ \Na_X^{\sE_H}(r \ot 1)} \phi\big) (s)
+ \frac{1}{4} c(\om(X)) T(s) 
\end{align}
for an arbitrary real vector field $X \in \sX(M)$.

Using this observation we compute the vertical part:
\[
\begin{split}
& V (S \otimes \ga_B)(\xi \ot r)
= i \sum_{j = 1}^{\dim (F)} V \big(  c_V(e_j) \nabla^{\sE_V}_{e_j}(\xi) \otimes \ga_B(r) \big) \\
& \q = i \sum_{j = 1}^{\dim (F)} c(e_j) V \big( \Na^{\sE_V}_{e_j}(\xi) \ot r \big) \\
& \q = i \sum_{j = 1}^{\dim (F)} c(e_j) \big( \Na^{\sE_M}_{e_j} + \frac{1}{4} c(\om(e_j)) \big) V(\xi \ot r) \, . 
\end{split}
\]
%

Using the observation in \eqref{eq:vert} one more time, we compute the horizontal part: 
\[
\begin{split}
& V(1 \ot_{\Na} D_B)(\xi \ot r) \\
& \q = i \sum_{\al = 1}^{\dim (B)} c( ( f_\al)_H) V\bigg( \xi \ot \Na_{f_\al}^{\sE_B}(r) 
+ \Na^{\sE_V}_{ (f_\al)_H}(\xi) \ot r \\
& \qqq + \frac{1}{2} k( (f_\al)_H) \cd (\xi \ot r) \bigg) \\
& \q = i \sum_{\al = 1}^{\dim (B)} c( (f_\al)_H) \big( \Na^{\sE_M}_{ (f_\al)_H} + \frac{1}{4} c( \om( (f_\al)_H)) \big)V (\xi \ot r) \\
& \qqq + \frac{i}{2} c(k^\sharp) V (\xi \ot r) \, .
\end{split}
\]

Applying Lemma \ref{lma:mean-curv} and the above computations, we see that the sum of the vertical and the horizontal part is given by
\[
\begin{split}
& \big( V (S \otimes \ga_B) + V(1 \ot_{\Na} D_B) \big) (\xi \ot r) \\
& \q = D_M V (\xi \ot r) + \frac{i}{2} c(k^\sharp)T (s) + \frac{i}{4}(c \ot c)(\sharp \ot 1)(\om) V(\xi \ot r) \\
& \q = D_M V (\xi \ot r) - \frac{i}{8}(c \ot c)(1 \ot \sharp)(\Om) V (\xi \ot r) .\qedhere
\end{split}
\]
%

We summarize the above results in the following: 
\begin{thm}
\label{thm:fact}
Suppose that $M \to B$ is a Riemannian submersion of even-dimensional spin$^c$-manifolds. Then the even spectral triple $(C^\infty(M), L^2(\sE_M),D_M)$ is the unbounded Kasparov product of the even unbounded Kasparov module $(C^\infty(M), X,S)$  with the even spectral triple $(C^\infty(B),L^2(\sE_B), D_B)$ up to the curvature term $-\frac{i}{8} (c \ot c)(1 \ot \sharp)(\Omega)$.
\end{thm}
\begin{proof}
%
We will show that the bounded transforms of the above unbounded Kasparov modules coincide with the shriek maps defined in \cite{CS84}. The result then follows from the wrong-way functoriality of the shriek map. Specifically, we know that the factorization $[M] \hot_{C(B)} \pi! = [B]$ holds in $KK$-theory ({\it cf.} Equation \eqref{eq:fact}). 

Recall \cite[Proposition 2.9]{CS84} that for any $K$-oriented submersion $f:X \to Y$ the class $f!$ can be represented by a continuous family of pseudodifferential operators of order $0$ (parametrized by $y \in Y$) on the vertical spinor bundle. For each $y \in Y$, the principal symbol is given by $\sigma(x,\xi) =  c_V(\xi^\sharp)( 1 + \|\xi^\sharp\|^2)^{-1/2}$ for any element $\xi \in \big( \ker (df)_x\big)^* \cong \Om^1_x\big( f^{-1}( \{y\}) \big)$, where $x \in f^{-1}(\{y\})$. In particular, if $X$ is a spin$^c$ manifold and $Y$ consists of a single point, the bounded transform $D_X(1+D_X^2)^{- 1/2}$ of the Dirac operator is a representative of the shriek map, which in this case is the fundamental class $[X]$ in $K$-homology. Moreover, from the construction of the unbounded Kasparov module $(C^\infty(M), X,S)$ we have that $S(1+S^2)^{-1/2}$ is a continuous family of pseudodifferential operators of order zero with the correct principal symbols, see the proof of Proposition \ref{p:unboukas}. We thus conclude that the spectral triples $(C^\infty(M), L^2(\sE_M),D_M)$ and $(C^\infty(B), L^2(\sE_B),D_B)$ represent the respective fundamental classes $[M] \in KK_0(C(M),\C)$ and $[B] \in KK_0(C(B),\C)$, and that the unbounded Kasparov module $(C^\infty(M), X, S)$ represents the class $\pi!$ in $KK_0(C(M),C(B))$. 
\end{proof}

\begin{rem}
As a special case of Theorem \ref{thm:fact} we can consider the following situation. Let $G$ be a simply-connected compact semisimple Lie group and $H$ a closed subgroup and consider $\pi: G \to G/H$. The Killing form on the corresponding Lie algebra $\mathfrak{g}$ then induces a Riemannian metric on $G$ and on $G/H$. In fact, we can identify 
\[
T_e(G) \cong \mathfrak{g} \cong \mathfrak{h} \oplus \mathfrak{h}^\perp, \qquad \text{and} \qquad T_{0}(G/H) \cong  \mathfrak{h}^\perp
\]
in terms of the Lie algebra $\mathfrak{h}$ of $H$. It is clear that then $\pi$ is a Riemannian submersion. Moreover, under the above assumptions on $G$ both $G$ and $G/H$ are spin manifolds (for a construction of Dirac operators $D_G$ and $D_{G/H}$ on $G$ and $G/H$ we refer to \cite{Bar92}). Hence, Theorem \ref{thm:fact} applies and yields a factorization of $D_G$ as a tensor sum
\[
D_G = S \otimes 1 + 1 \otimes_\nabla D_{G/H},
\]
up to an explicit curvature-term. In particular, this applies to the Hopf fibration $\pi:\S^3 \to \S^2$ for which the explicit factorization was already obtained in \cite{BMS13}. In fact, one can compute that 
\[
\Omega( f_1, f_2,e_1) = 2
\]
in terms of the vertical $e_1$ and horizontal vector fields $f_1,f_2$. The curvature term appearing in Theorem \ref{thm:fact} above is then
$$
- \frac{i}{8} (c \ot c)(1 \ot \sharp)(\Omega) =-\frac i 4 [\gamma_2,\gamma_3]  \gamma_1= -\frac 12,
$$
where we have expressed $c(e_1), c(f_1)$ and $c(f_2)$ in terms of the Pauli matrices $\gamma_1,\gamma_2,\gamma_3$, with the convention used in \cite{BMS13}.
This is an independent check for the $-\frac{1}{2}$ appearing in \cite[Theorem 6.31]{BMS13}, at the same time giving meaning to it as a curvature term. More generally, such homogeneous spaces are subject of study in \cite{CM16}.
\end{rem}

\newcommand{\noopsort}[1]{}\def\cprime{$'$}

\end{document}